\documentclass[a4paper,11pt]{article}
\usepackage{amsmath, amsfonts, amscd, amssymb, amsthm, enumerate, xypic}

\def\bfB{\mathbf{B}}

\newcommand{\Mat}{\operatorname{M}}

\newcommand{\Mata}{\operatorname{A}}

\newcommand{\GL}{\operatorname{GL}}
\newcommand{\Ker}{\operatorname{Ker}}

\newcommand{\Vect}{\operatorname{span}}
\newcommand{\im}{\operatorname{Im}}
\newcommand{\ad}{\operatorname{ad}}

\newcommand{\tr}{\operatorname{tr}}

\newcommand{\rk}{\operatorname{rk}}
\newcommand{\codim}{\operatorname{codim}}
\renewcommand{\setminus}{\smallsetminus}

\def\K{\mathbb{K}}

\def\calB{\mathcal{B}}
\def\calC{\mathcal{C}}
\def\calD{\mathcal{D}}

\def\calG{\mathcal{G}}
\def\calH{\mathcal{H}}

\def\calP{\mathcal{P}}

\def\calV{\mathcal{V}}

\def\calX{\mathcal{X}}

\def\lcro{\mathopen{[\![}}
\def\rcro{\mathclose{]\!]}}

\theoremstyle{definition}
\newtheorem{Def}{Definition}

\theoremstyle{plain}
\newtheorem{theo}{Theorem}
\newtheorem{prop}[theo]{Proposition}
\newtheorem{cor}[theo]{Corollary}
\newtheorem{lemma}[theo]{Lemma}

\theoremstyle{plain}

\theoremstyle{remark}
\newtheorem{Rems}{Remarks}
\newtheorem{Rem}[Rems]{Remark}

\title{Commutators from a hyperplane of matrices}
\author{Cl\'ement de Seguins Pazzis\footnote{
Universit\'e de Versailles Saint-Quentin-en-Yvelines, Laboratoire de Math\'ematiques
de Versailles, 45 avenue des Etats-Unis, 78035 Versailles cedex, France}
\footnote{e-mail address: dsp.prof@gmail.com}}

\begin{document}

\thispagestyle{plain}

\maketitle

\begin{abstract}
Denote by $\Mat_n(\K)$ the algebra of $n$ by $n$ matrices with entries in the field $\K$.
A theorem of Albert and Muckenhoupt states that every trace zero
matrix of $\Mat_n(\K)$ can be expressed as $AB-BA$ for some pair $(A,B)\in \Mat_n(\K)^2$.
Assuming that $n>2$ and that $\K$ has more than $3$ elements,
we prove that the matrices $A$ and $B$ can be required to belong to an arbitrary given hyperplane of $\Mat_n(\K)$.
\end{abstract}

\vskip 2mm
\noindent
\emph{AMS Classification:} 15A24, 15A30

\vskip 2mm
\noindent
\emph{Keywords:} commutator; trace; hyperplane; matrices

\vskip 4mm

\section{Introduction}

\subsection{The problem}

In this article, we let $\K$ be an arbitrary field.
We denote by $\Mat_n(\K)$ the algebra of square matrices with $n$ rows and entries in $\K$,
and by $\frak{sl}_n(\K)$ its hyperplane of trace zero matrices. The trace of a matrix $M \in \Mat_n(\K)$
is denoted by $\tr M$. Given two matrices $A$ and $B$ of $\Mat_n(\K)$, one sets
$$[A,B]:=AB-BA,$$
known as the commutator, or Lie bracket, of $A$ and $B$. Obviously, $[A,B]$ belongs to $\frak{sl}_n(\K)$.
Although it is easy to see that the linear subspace spanned by the commutators is $\frak{sl}_n(\K)$,
it is more difficult to prove that every trace zero matrix is actually a commutator,
a theorem which was first proved by Shoda \cite{Shoda} for fields of characteristic $0$,
and later generalized to all fields by Albert and Muckenhoupt \cite{AlbertMuck}. Recently, exciting new developments
on this topic have appeared: most notably, the long-standing conjecture that the result holds for all principal ideal domains
has just been solved by Stasinski \cite{Stasinski} (the case of integers had been worked out earlier by Laffey and Reams \cite{LaffeyReams}).

Here, we shall consider the following variation of the above problem:
\begin{center}Given a (linear) hyperplane $\calH$ of $\Mat_n(\K)$,
is it true that every trace zero matrix is the commutator of two matrices of $\calH$?
\end{center}

Our first motivation is that this constitutes a natural generalization of the following result of Thompson:

\begin{theo}[Thompson, Theorem 5 of \cite{Thompson}]\label{hypercan}
Assume that $n \geq 3$. Then,
$[\frak{sl}_n(\K),\frak{sl}_n(\K)]=\frak{sl}_n(\K)$.
\end{theo}

Another motivation stems from the following known theorem:

\begin{theo}[Proposition 4 of \cite{dSPlargedimprod}]
Let $\calV$ be a linear subspace of $\Mat_n(\K)$ with $\codim \calV<n-1$.
Then, $\frak{sl}_n(\K)=\Vect \bigl\{[A,B] \mid (A,B) \in \calV^2\bigr\}$.
\end{theo}

Thus, a natural question to ask is whether, in the above situation, every trace zero matrix is a commutator of two matrices of $\calV$.
Studying the case of hyperplanes is an obvious first step in that direction (and a rather non-trivial one, as we shall see).

An additional motivation is the corresponding result for products (instead of commutators) that we have obtained
in \cite{dSPlargedimprod}:

\begin{theo}[Theorem 3 of \cite{dSPlargedimprod}]
Let $\calH$ be a (linear) hyperplane of $\Mat_n(\K)$, with $n>2$. Then, every matrix of $\Mat_n(\K)$
splits up as $AB$ for some $(A,B) \in \calH^2$.
\end{theo}

\subsection{Main result}

In the present paper, we shall prove the following theorem:

\begin{theo}\label{dSPcrochet}
Assume that $\# \K>3$ and $n>2$. Let $\calH$ be an arbitrary hyperplane of $\Mat_n(\K)$.
Then, every trace zero matrix of $\Mat_n(\K)$ splits up as $AB-BA$ for some $(A,B)\in \calH^2$.
\end{theo}

Let us immediately discard an easy case. Assume that $\calH$ does not contain the identity matrix $I_n$.
Then, given $(A,B)\in \Mat_n(\K)^2$, we have
$$[\lambda I_n+A,\mu I_n+B]=[A,B]$$
for all $(\lambda,\mu)\in \K^2$, and obviously there is a unique pair $(\lambda,\mu)\in \K^2$ such that
$\lambda I_n+A$ and $\mu I_n+B$ belong to $\calH$. In that case, it follows from the Albert-Muckenhoupt theorem
that every matrix of $\frak{sl}_n(\K)$ is a commutator of matrices of $\calH$.
Thus, the only case left to consider is the one when $I_n \in \calH$. As we shall see, this is a highly non-trivial problem.
Our proof will broadly consist in refining Albert and Muckenhoupt's method.

\paragraph{}
The case $n=2$ can be easily described over any field:

\begin{prop}
Let $\calH$ be a hyperplane of $\Mat_2(\K)$.
\begin{enumerate}[(a)]
\item If $\calH$ contains $I_2$, then $[\calH,\calH]$ is a $1$-dimensional linear subspace of $\Mat_2(\K)$.
\item If $\calH$ does not contain $I_2$, then $[\calH,\calH]=\frak{sl}_2(\K)$.
\end{enumerate}
\end{prop}

\begin{proof}
Point (b) has just been explained. Assume now that $I_2 \in \calH$. Then, there are matrices $A$ and $B$ such that
$(I_2,A,B)$ is a basis of $\calH$. For all $(a,b,c,a',b',c')\in \K^6$, one finds
$$[aI_2+bA+cB\,, \,a'I_2+b'A+c'B]=(bc'-b'c)[A,B].$$
Moreover, as $A$ is a $2 \times 2$ matrix and not a scalar multiple of the identity, it is similar to a companion matrix,
whence the space of all matrices which commute with $A$ is $\Vect(I_2,A)$. This yields $[A,B] \neq 0$.
As obviously $\K=\bigl\{bc'-b'c\mid (b,c,b',c')\in \K^4\bigr\}$, we deduce that $[\calH,\calH]=\K\, [A,B]$
with $[A,B] \neq 0$.
\end{proof}

\subsection{Additional definitions and notation}

\begin{itemize}
\item Given a subset $\calX$ of $\Mat_n(\K)$, we set
$$[\calX,\calX]:=\bigl\{[A,B] \mid (A,B)\in \calX^2\bigr\}.$$

\item The canonical basis of $\K^n$ is denoted by $(e_1,\dots,e_n)$.

\item Given a basis $\calB$ of $\K^n$, the matrix of coordinates of $\calB$ in the canonical basis of $\K^n$
is denoted by $P_\calB$.

\item Given $i$ and $j$ in $\lcro 1,n\rcro$, one denotes by $E_{i,j}$ the matrix of $\Mat_n(\K)$ with all entries zero
except the one at the $(i,j)$-spot, which equals $1$.

\item A matrix of $\Mat_n(\K)$ is \textbf{cyclic} when its minimal polynomial has degree $n$ or, equivalently,
when it is similar to a companion matrix.

\item
The $n$ by $n$ nilpotent Jordan matrix
is denoted by
$$J_n=\begin{bmatrix}
0 & 1 & & (0) \\
& \ddots & \ddots & \\
& & \ddots & 1 \\
(0) & & & 0
\end{bmatrix}.$$

\item A Hessenberg matrix is a square matrix $A=(a_{i,j}) \in \Mat_n(\K)$ in which $a_{i,j}=0$ whenever $i>j+1$.
In that case, we set
$$\ell(A):=\bigl\{j \in \lcro 1,n-1\rcro : \; a_{j+1,j} \neq 0\bigr\}.$$

\item One equips $\Mat_n(\K)$ with the non-degenerate symmetric bilinear form
$$b : (M,N) \mapsto \tr(MN),$$
to which orthogonality refers in the rest of the article.
\end{itemize}
Given $A \in \Mat_n(\K)$, one sets
$$\ad_A : M \in \Mat_n(\K) \mapsto [A,M] \in \Mat_n(\K),$$
which is an endomorphism of the vector space $\Mat_n(\K)$; its kernel is the centralizer
$$\calC(A):=\bigl\{M \in \Mat_n(\K) : AM=MA\bigr\}$$
of the matrix $A$.
Recall the following nice description of the range of $\ad_A$, which follows from the rank theorem and the basic observation that $\ad_A$ is skew-symmetric for the bilinear form $(M,N) \mapsto \tr(MN)$:

\begin{lemma}\label{imad}
Let $A \in \Mat_n(\K)$. The range of $\ad_A$ is the orthogonal of $\calC(A)$, that is the set of all $N \in \Mat_n(\K)$ for which
$$\forall B \in \calC(A), \; \tr(B\,N)=0.$$
\end{lemma}

In particular, if $A$ is cyclic then its centralizer is $\K[A]=\Vect(I_n,A,\dots,A^{n-1})$, whence
$\im (\ad_A)$ is defined by a set of $n$ linear equations:

\begin{lemma}\label{imadcyclic}
Let $A \in \Mat_n(\K)$ be a cyclic matrix. The range of $\ad_A$ is the set of all $N \in \Mat_n(\K)$ for which
$$\forall k \in \lcro 0,n-1\rcro, \; \tr(A^k\,N)=0.$$
\end{lemma}

\begin{Rem}\label{specialcasesremark}
Interestingly, the two special cases below yield the strategy for Shoda's approach and Albert and Muckenhoupt's, respectively:
\begin{enumerate}[(i)]
\item Let $D$ be a diagonal matrix of $\Mat_n(\K)$ with distinct diagonal entries. Then, the centralizer of $D$
is the space $\calD_n(\K)$ of all diagonal matrices, and hence $\im \ad_D$ is the space of all matrices with diagonal zero.
As every trace zero matrix that is not a scalar multiple of the identity is similar to a matrix with diagonal zero
\cite{Fillmore}, Shoda's theorem of \cite{Shoda} follows easily.

\item Consider the case of the Jordan matrix $J_n$. As $J_n$ is cyclic, Lemma \ref{imadcyclic} yields that $\im (\ad_{J_n})$
is the set of all matrices $A=(a_{i,j}) \in \Mat_n(\K)$ for which $\underset{k=1}{\overset{n-\ell}{\sum}} a_{k+\ell,k}=0$ for all $\ell \in \lcro 0,n-1\rcro$.
In particular, if $A=(a_{i,j}) \in \Mat_n(\K)$ is Hessenberg, then this condition is satisfied whenever $\ell>1$,
and hence $A \in \im (\ad_{J_n})$ if and only if $\tr A=0$ and $\underset{k=1}{\overset{n-1}{\sum}} a_{k+1,k}=0$.
Albert and Muckenhoupt's proof is based upon the fact that, except for a few special cases, the similarity class of a matrix must contain a Hessenberg
matrix $A$ that satisfies the extra equation $\underset{k=1}{\overset{n-1}{\sum}} a_{k+1,k}=0$.
\end{enumerate}
\end{Rem}

\section{Proof of the main theorem}\label{finalproofsection}

\subsection{Proof strategy}\label{proofstrategy}

Let $\calH$ be a hyperplane of $\Mat_n(\K)$. We already know that $[\calH,\calH]=\mathfrak{sl}_n(\K)$ if $I_n \not\in \calH$.
Thus, in the rest of the article, we will only consider the case when $I_n \in \calH$.

Our proof will use three basic but potent principles:

\begin{enumerate}[(1)]
\item Given $A \in \mathfrak{sl}_n(\K)$, if some $A_1 \in \calH$
satisfies $A \in \im (\ad_{A_1})$ and $\calC(A_1) \not\subset \calH$, then $A \in [\calH,\calH]$.
Indeed, in that situation, we find $A_2 \in \Mat_n(\K)$ such that
$A=[A_1,A_2]$, together with some $A_3 \in \calC(A_1)$ for which $A_3 \not\in \calH$.
Then, the affine line $A_2+\K A_3$ is included in the inverse image of $\{A\}$ by
$\ad_{A_1}$ and it has exactly one common point with $\calH$.

\item Let $(A,B)\in \mathfrak{sl}_n(\K)^2$ and $\lambda \in \K$. If there are matrices $A_1$ and $A_2$
such that $A=[A_1,A_2]$ and $\tr(B\,A_1)=\tr(B\,A_2)=0$, then we also have $\tr((B-\lambda\,A)A_1)=\tr((B-\lambda\, A)A_2)=0$. \\
Indeed, equality $A=[A_1,A_2]$ ensures that $\tr(A\,A_1)=\tr(A\,A_2)=0$ (see Lemma \ref{imad}).

\vskip 2mm
\item Let $(A,B)\in \Mat_n(\K)^2$ and $P \in \GL_n(\K)$.
Setting $\calG:=\{B\}^\bot$, we see that the assumption $A \in [\calG,\calG]$ implies
$PAP^{-1} \in [P\calG P^{-1},P\calG P^{-1}]$, while $P\calG P^{-1}=\{PBP^{-1}\}^\bot$.
\end{enumerate}

Now, let us give a rough idea of the proof strategy.
One fixes $A \in \mathfrak{sl}_n(\K)$ and aims at proving that $A \in [\calH,\calH]$.
We fix a non-zero matrix $B$ such that $\calH=\{B\}^\bot$.

Our basic strategy is the Albert-Muckenhoupt method: we try to find a cyclic matrix $M$ in $\calH$
such that $A \in \im(\ad_M)$; if $A \not\in \ad_M(\calH)$, then we learn that $\calC(M) \subset \calH$
(see principle (1) above), which yields additional information on $B$.
Most of the time, we will search for such a cyclic matrix $M$ among the nilpotent matrices with rank $n-1$.
The most favorable situation is the one where $A$ is either upper-triangular or Hessenberg with enough non-zero sub-diagonal entries: in these cases, we search for a good matrix $M$ among the strictly upper-triangular matrices with rank $n-1$ (see Lemma \ref{prelimlemma}). If this method yields no solution, then we learn precious information on the
simultaneous reduction of the endomorphisms $X \mapsto AX$ and $X \mapsto BX$.
Using changes of bases, we shall see that either the above method delivers a solution for a pair $(A',B')$
that is simultaneously similar to $(A,B)$, in which case Principle (3) shows that we have a solution for $(A,B)$,
or $(I_n,A,B)$ is locally linearly dependent (see the definition below), or else $n=3$ and $A$ is similar to $\lambda I_3+E_{2,3}$ for some $\lambda \in \K$.
When $(I_n,A,B)$ is locally linearly dependent and $A$ is not of that special type, one uses the classification of
locally linearly dependent triples to reduce the situation to the one where $B=I_n$, that is
$\calH=\mathfrak{sl}_n(\K)$, and in that case the proof is completed by invoking Theorem \ref{hypercan}.
Finally, the case when $A$ is similar to $\lambda I_3+E_{2,3}$ for some $\lambda \in \K$ will be dealt with independently (Section \ref{specialcasesection})
by applying Albert and Muckenhoupt's method for well-chosen companion matrices instead of a Jordan nilpotent matrix.

Let us finish these strategic considerations by recalling the notion of local linear dependence:

\begin{Def}
Given vector spaces $U$ and $V$, linear maps $f_1,\dots,f_n$ from $U$ to $V$ are called locally linearly dependent (in abbreviated form: LLD)
when the vectors $f_1(x),\dots,f_n(x)$ are linearly dependent for all $x \in U$.
\end{Def}

We adopt a similar definition for matrices by referring to the linear maps that are canonically associated with these matrices.

\subsection{The basic lemma}

\begin{lemma}\label{prelimlemma}
Let $(A,B) \in \frak{sl}_n(\K)^2$ be with $B=(b_{i,j}) \neq 0$, and set $\calH:=\{B\}^\bot$.
In each one of the following cases, $A$ belongs to $[\calH,\calH]$:
\begin{enumerate}[(a)]
\item $\# \K>2$, $A$ is upper-triangular and $B$ is not Hessenberg.
\item $\# \K>3$, $A$ is Hessenberg and there exist $i \in \lcro 2,n-1\rcro$ and $j \in \lcro 3,n\rcro \setminus \{i\}$ such that $\{1,i\} \subset \ell(A)$ and $b_{j,1} \neq 0$.
\end{enumerate}
\end{lemma}

\begin{proof}
We use a \emph{reductio ad absurdum}, assuming that $A \not\in [\calH,\calH]$. We write $A=(a_{i,j})$.

\begin{enumerate}[(a)]
\item Assume that $\# \K>2$, that $A$ is upper-triangular and that $B$ is not Hessenberg.
We choose a pair $(l,l')\in \lcro 1,n\rcro^2$ such that $b_{l,l'}\neq 0$, with
$l-l'$ maximal for such pairs. Thus, $l-l'>1$.
Let $(x_1,\dots,x_{n-1}) \in (\K^*)^{n-1}$, and set
$$\beta:=\frac{\underset{k=1}{\overset{n-1}{\sum}} b_{k+1,k}\,x_k}{b_{l,l'}} \quad \text{and} \quad M:=\underset{k=1}{\overset{n-1}{\sum}}x_k\,E_{k,k+1}-\beta\,E_{l',l.}$$
We see that $M$ is nilpotent of rank $n-1$, and hence it is cyclic.
One notes that $M \in \calH$. Moreover, $\tr(A M^k)=0$ for all $k \geq 1$, because $A$ is upper-triangular and $M$
is strictly upper-triangular, whereas $\tr(A)=0$ by assumption. Thus, $A \in \im (\ad_M)$.
As it is assumed that $A \not\in \ad_M(\calH)$, one deduces from principle (1) in Section \ref{proofstrategy}
that $\calC(M) \subset \calH$; in particular $\tr(M^{l-l'} B)=0$, which, as $b_{i,j}=0$ whenever $i-j>l-l'$, reads
$$b_{l-l'+1,1}\,x_1x_2 \cdots x_{l-l'}+b_{l-l'+2,2}\,x_2x_3\cdots x_{l-l'+1}+\cdots +
b_{n,n-l+l'}\,x_{n-l+l'}\cdots x_{n-1}=0.$$
Here, we have a polynomial with degree at most $1$ in each variable $x_i$, and this polynomial vanishes
at every $(x_1,\dots,x_{n-1}) \in (\K^*)^{n-1}$, with $\# \K^* \geq 2$.
It follows that $b_{i,j}=0$ for all $(i,j)\in \lcro 1,n\rcro^2$ with $i-j=l-l'$, and the special
case $(i,j)=(l,l')$ yields a contradiction.

\vskip 2mm
\item Now, we assume that $\# \K>3$, that $A$ is Hessenberg and that there exist $i \in \lcro 2,n\rcro$ and
$j \in \lcro 3,n\rcro \setminus \{i\}$ such that $\{1,i\} \subset \ell(A)$ and $b_{j,1} \neq 0$.
The proof strategy is similar to the one of case (a), with additional technicalities.
One chooses a pair $(l,l')\in \lcro 1,n\rcro^2$ such that $b_{l,l'} \neq 0$, with $l-l'$ maximal
for such pairs (again, the assumptions yield $l-l'\geq j-1>1$). As $a_{2,1} \neq 0$, no generality is lost in assuming that $a_{2,1}=1$.
We introduce the formal polynomial
$$\mathbf{p}:=\underset{k=1}{\overset{n-2}{\sum}}a_{k+2,k+1}\,\mathbf{x}_k \in \K[\mathbf{x}_1,\mathbf{x}_2,\dots,\mathbf{x}_{n-2}].$$
Let $(x_1,\dots,x_{n-2}) \in (\K^*)^{n-2}$,
and set
$$\alpha:=\mathbf{p}(x_1,\dots,x_{n-2}) \quad \text{and} \quad
\beta:=\frac{\alpha\,b_{2,1}- \underset{k=1}{\overset{n-2}{\sum}}x_k \, b_{k+2,k+1}}{b_{l,l'}}\cdot$$
Finally, set
$$M:=-\alpha\,E_{1,2}+\underset{k=1}{\overset{n-2}{\sum}}x_k\,E_{k+1,k+2}+\beta\,E_{l',l}.$$
The definition of $M$ shows that $\tr(MA)=\tr(MB)=0$, and in particular $M \in \calH$.
Assume now that $\mathbf{p}(x_1,\dots,x_{n-2}) \neq 0$. Then, $M$ is cyclic as it is nilpotent with rank $n-1$.
As $A$ is Hessenberg, we also see that $\tr(M^k\,A)=0$ for all $k \geq 2$.
Thus, $\tr(M^kA)=0$ for every non-negative integer $k$, and hence Lemma \ref{imadcyclic} yields $A \in \im (\ad_M)$.
It ensues that $\calC(M) \subset \calH$, and in particular $\tr(M^{j-1} B)=0$.
As $l-l'>1$, we see that, for all $(a,b)\in \lcro 1,n\rcro^2$ with $b-a \leq l-l'$,
and every integer $c>1$, the matrices $M^c$ and $\Bigl(-\alpha\,E_{1,2}+\underset{k=1}{\overset{n-2}{\sum}}x_k\,E_{k+1,k+2}\Bigr)^c$
have the same entry at the $(a,b)$-spot; in particular, for all $k \in \lcro 2,n-j+1\rcro$, the entry of
$M^{j-1}$ at the $(k,j+k-1)$-spot is $x_{k-1}x_k \cdots x_{k-3+j}$, and the entry of $M^{j-1}$ at the $(1,j)$-spot is
$-\alpha\, x_1\cdots x_{j-2}$; moreover, for all $(a,b)\in \lcro 1,n\rcro^2$ with $b-a \leq \ell-\ell'$ and $b-a \neq j-1$,
the entry of $M^{j-1}$ at the $(a,b)$-spot is $0$.
Therefore, equality $\tr(M^{j-1}B)=0$ yields
$$-b_{j,1}\,\alpha\,x_1\cdots x_{j-2}+b_{j+1,2}\,x_1\cdots x_{j-1}+
b_{j+2,3}\,x_2\cdots x_j+\cdots+b_{n,n-j+1}\,x_{n-j}\cdots x_{n-2}=0.$$
We conclude that we have established the following identity:
for the polynomial
$$\mathbf{q}:=\mathbf{p}\times \Bigl(-b_{j,1}\,\mathbf{p}\,\mathbf{x}_1\cdots \mathbf{x}_{j-2}+b_{j+1,2}\,\mathbf{x}_1\cdots \mathbf{x}_{j-1}+
b_{j+2,3}\,\mathbf{x}_2\cdots \mathbf{x}_j+\cdots+b_{n,n-j+1}\,\mathbf{x}_{n-j}\cdots \mathbf{x}_{n-2}\Bigr),$$
we have
$$\forall (x_1,\dots,x_{n-2}) \in (\K^*)^{n-2}, \quad
\mathbf{q}(x_1,\dots,x_{n-2})=0.$$
Noting that $\mathbf{q}$ has degree at most $3$ in each variable, we split the discussion into two main cases.

\textbf{Case 1. $\# \K>4$.} \\
Then, $\# \K^*> 3$ and hence $\mathbf{q}=0$.
As $\mathbf{p} \neq 0$ (remember that $a_{i+1,i}\neq 0$),
it follows that
$$-b_{j,1}\,\mathbf{p}\,\mathbf{x}_1\cdots \mathbf{x}_{j-2}+b_{j+1,2}\,\mathbf{x}_1\cdots \mathbf{x}_{j-1}+
b_{j+2,3}\,\mathbf{x}_2\cdots \mathbf{x}_j+\cdots+b_{n,n-j+1}\,\mathbf{x}_{n-j}\cdots \mathbf{x}_{n-2}=0.$$
As $b_{j,1} \neq 0$, identifying the coefficients of the monomials of type
$\mathbf{x}_1\cdots \mathbf{x}_{j-2}\mathbf{x}_k$ with $k \in \lcro 1,n-2\rcro \setminus \{j-1\}$
leads to $a_{k+2,k+1}=0$ for all such $k$. This contradicts the assumption that $a_{i+1,i} \neq 0$.

\vskip 2mm
\textbf{Case 2. $\# \K=4$.} \\
A polynomial of $\K[t]$ which vanishes at every non-zero element of $\K$ must be a multiple of $t^3-1$.
In particular, if such a polynomial has degree at most $3$, we may write it as
$\alpha_3\,t^3+\alpha_2\,t^2+\alpha_1\,t+\alpha_0$, and we obtain
$\alpha_3=-\alpha_0$. From there, we split the discussion into two subcases.

\noindent \textbf{Subcase 2.1.} $i >j$. \\
Then, $\mathbf{q}$ has degree at most $2$ in $\mathbf{x}_{i-1}$.
Thus, if we see $\mathbf{q}$ as a polynomial in the sole variable $\mathbf{x}_{i-1}$,
the coefficients of this polynomial must vanish for every specialization of $\mathbf{x}_1,\dots,\mathbf{x}_{i-2},\mathbf{x}_i,\dots,\mathbf{x}_{n-2}$ in $\K^*$;
extracting the coefficients of $(\mathbf{x}_{i-1})^2$ leads to the identity
$$\forall (x_1,\dots,x_{i-2},x_i,\dots,x_{n-2})\in (\K^*)^{n-3}, \quad
-b_{j,1} (a_{i+1,i})^2 \,x_1\cdots x_{j-2}+\mathbf{r}(x_1,\dots,x_{n-2})=0$$
where $\mathbf{r}=\underset{k=i-j+1}{\overset{n-j}{\sum}}  a_{i+1,i}\,b_{j+k,k+1}\, \mathbf{x}_k \cdots \mathbf{x}_{i-2} \mathbf{x}_{i} \cdots \mathbf{x}_{j-2+k}$.
Noting that the degree of $-b_{j,1} (a_{i+1,i})^2\, \mathbf{x}_1\cdots \mathbf{x}_{j-2}+\mathbf{r}$
is at most $1$ in each variable, we deduce that this polynomial is zero.
This contradicts the fact that the coefficient of $\mathbf{x}_1\cdots \mathbf{x}_{j-2}$ is
$-b_{j,1} (a_{i+1,i})^2$, which is non-zero according to our assumptions.

\noindent \textbf{Subcase 2.2.} $i<j$. \\
Let us fix $x_1,\dots,x_{i-2},x_i,\dots,x_{n-2}$ in $\K^*$.
The coefficient of $\mathbf{q}(x_1,\dots,x_{i-2},\mathbf{x}_{i-1},x_i,\dots,x_{n-2})$ with respect to $(\mathbf{x}_{i-1})^3$ is
$$-b_{j,1}(a_{i+1,i})^2\,x_1\cdots x_{i-2}x_i\cdots x_{j-2}.$$
One the other hand, with
$$\mathbf{s}:=\underset{i \leq k \leq n-j}{\sum}b_{j+k,k+1} \underset{\ell=k}{\overset{j-2+k}{\prod}} \mathbf{x}_\ell,$$
the coefficient of $\mathbf{q}(x_1,\dots,x_{i-2},\mathbf{x}_{i-1},x_i,\dots,x_{n-2})$ with respect to $(\mathbf{x}_{i-1})^0$ is
$$\mathbf{s}(x_1,\dots,x_{i-2},x_i,\dots,x_{n-2})\,\underset{k \in \lcro 1,n-2\rcro \setminus \{i-1\}}{\sum}\,a_{k+2,k+1}\, x_k.$$
Therefore,
\begin{multline*}
\forall (x_1,\dots,x_{n-2})\in (\K^*)^{n-2}, \quad \\
b_{j,1}(a_{i+1,i})^2\,x_1\cdots x_{i-2}x_i\cdots x_{j-2}
=\mathbf{s}(x_1,\dots,x_{i-2},x_i,\dots,x_{n-2}) \\
\times \underset{k \in \lcro 1,n-2\rcro \setminus \{i-1\}}{\sum}\,a_{k+2,k+1}\, x_k.
\end{multline*}
On both sides of this equality, we have polynomials of degree at most $2$ in each variable.
As $\# (\K^*) >2$, we deduce the identity
$$b_{j,1}(a_{i+1,i})^2\,\mathbf{x}_1\dots \mathbf{x}_{i-2}\mathbf{x}_i\cdots \mathbf{x}_{j-2}
=\mathbf{s}\times \underset{k \in \lcro 1,n-2\rcro \setminus \{i-1\}}{\sum}\,a_{k+2,k+1}\, \mathbf{x}_k.$$
However, on the left-hand side of this identity is a non-zero homogeneous polynomial of degree $j-3$,
whereas its right-hand side is a homogeneous polynomial of degree $j$.
There lies a final contradiction.
\end{enumerate}
\end{proof}

\subsection{Reduction to the case when $I_n,A,B$ are locally linearly dependent}

In this section, we use Lemma \ref{prelimlemma} to prove the following result:

\begin{lemma}\label{reductionlemma}
Assume that $\# \K > 3$, let $(A,B) \in \frak{sl}_n(\K)^2$ be such that $B \neq 0$, and set
$\calH:=\{B\}^\bot$.
Then, either $A \in [\calH,\calH]$, or $(I_n,A,B)$ is LLD, or
$A$ is similar to $\lambda I_3+E_{2,3}$ for some $\lambda \in \K$.
\end{lemma}

In order to prove Lemma \ref{reductionlemma}, one needs two preliminary results.
The first one is a basic result in the theory of matrix spaces with rank bounded above.

\begin{lemma}[Lemma 2.4 of \cite{dSPLLD1}]\label{decompositionlemma}
Let $m,n,p,q$ be positive integers, and $\calV$ be a linear subspace of $\Mat_{m+p,n+q}(\K)$
in which every matrix splits up as
$$M=\begin{bmatrix}
A(M) & [?]_{m \times q} \\
[0]_{p \times n} & B(M)
\end{bmatrix}$$
where $A(M)  \in \Mat_{m,n}(\K)$ and $B(M) \in \Mat_{p,q}(\K)$.
Assume that there is an integer $r$ such that $\forall M \in \calV, \; \rk M \leq r <\# \K$,
and set $s:=\max \{\rk A(M) \mid M \in \calV\}$ and $t:=\max \{\rk B(M) \mid M \in \calV\}$.
Then, $s+t \leq r$.
\end{lemma}

\begin{lemma}\label{eigenvectorslemma}
Assume that $\# \K \geq 3$. Let $V$ be a vector space over $\K$ and $u$ be an endomorphism of $V$
that is not a scalar multiple of the identity. Then, there are two linearly independent non-eigenvectors of $u$.
\end{lemma}

\begin{proof}[Proof of Lemma \ref{eigenvectorslemma}]
As $u$ is not a scalar multiple of the identity, some vector $x \in V \setminus \{0\}$
is not an eigenvector of $u$. Then, the $2$-dimensional subspace $P:=\Vect(x,u(x))$ contains $x$.
As $u_{|P}$ is not a scalar multiple of the identity, $u$ stabilizes at most two
$1$-dimensional subspaces of $P$. As $\# \K>2$, there are at least four $1$-dimensional subspaces of $P$, whence at least two of them are not stable under $u$. This proves our claim.
\end{proof}

Now, we are ready to prove Lemma \ref{reductionlemma}.

\begin{proof}[Proof of Lemma \ref{reductionlemma}]
Throughout the proof, we assume that $A \not\in [\calH,\calH]$ and that there is no scalar
$\lambda$ such that $A$ is similar to $\lambda I_3+E_{2,3}$. Our aim is to show that
$(I_n,A,B)$ is LLD.

Note that, for all $P \in \GL_n(\K)$, no pair $(M,N)\in \Mat_n(\K)^2$ satisfies both
$[M,N]=P^{-1}AP$ and $\tr((P^{-1}BP) M)=\tr((P^{-1}BP) N)=0$.

Let us say that a vector $x \in \K^n$
has \textbf{order $3$} when $\rk(x,Ax,A^2x)=3$.
Let $x \in \K^n$ be of order $3$. Then, $(x,Ax,A^2x)$ may be extended into a basis
$\bfB=(x_1,x_2,x_3,x_4,\dots,x_n)$ of $\K^n$
such that $A':=P_\bfB^{-1}AP_\bfB$ is Hessenberg\footnote{One finds such a basis by induction as follows:
one sets $(x_1,x_2,x_3):=(x,Ax,A^2x)$ and, given $k \in \lcro 4,n\rcro$ such that $x_1,\dots,x_{k-1}$ are defined,
one sets $x_k:=Ax_{k-1}$ if $Ax_{k-1}\not\in \Vect(x_1,\dots,x_{k-1})$, otherwise one chooses an arbitrary vector
$x_k \in \K^n \setminus \Vect(x_1,\dots,x_{k-1})$.}. Moreover, one sees that $\{1,2\} \subset \ell(A')$.
Applying point (a) of Lemma \ref{prelimlemma}, one obtains that the entries in the first column of
$P^{-1}_\bfB\,B\,P_\bfB$ are all zero starting from the third one, which means that $Bx \in \Vect(x,Ax)$.

Let now $x \in \K^n$ be a vector that is not of order $3$. If $x$ and $Ax$ are linearly dependent, then
$x$, $Ax$, $Bx$ are linearly dependent. Thus, we may assume that $\rk(x,Ax)=2$ and $A^2 x \in \Vect(x,Ax)$.
We split $\K^n=\Vect(x,Ax)\oplus F$ and we choose a basis $(f_3,\dots,f_n)$ of $F$.
For $\bfB:=(x,Ax,f_3,\dots,f_n)$, we now have, for some $(\alpha,\beta)\in \K^2$ and some $N \in \Mat_{n-2}(\K)$,
$$P_{\bfB}^{-1}\,A\,P_\bfB=\begin{bmatrix}
K & [?]_{2 \times (n-2)} \\
[0]_{(n-2) \times 2} &  N
\end{bmatrix} \quad \text{where
$K=\begin{bmatrix}
0 & \alpha \\
1 & \beta
\end{bmatrix}$.}$$
From there, we split the discussion into several cases, depending on the form of $N$ and its relationship with $K$.

\noindent \textbf{Case 1.} $N \not\in \K I_{n-2}$. \\
Then, there is a vector $y \in \K^{n-2}$ for which $y$ and $Ny$ are linearly independent.
Denoting by $z$ the vector of $F$ with coordinate list $y$ in $(f_3,\dots,f_n)$,
one obtains $\rk(x,Ax,z,Az)=4$, and hence one may extend $(x,Ax,z,Az)$
into a basis $\bfB'$ of $\K^n$ such that $A':=P_{\bfB'}^{-1}AP_{\bfB'}$ is Hessenberg
with $\{1,3\} \subset \ell(A')$. Point (b) of Lemma \ref{prelimlemma} shows that, in the first column of
$P_{\bfB'}^{-1}BP_{\bfB'}$, all the entries must be zero starting from the fourth one,
yielding $Bx \in \Vect(x,Ax,z)$.
As $N \not\in \K I_{n-2}$, we know from Lemma \ref{eigenvectorslemma} that we may find another vector $z' \in F \setminus \K z$
such that $\rk(x,Ax,z',Az')=4$, which yields $Bx \in \Vect(x,Ax,z')$.
Thus, $Bx \in \Vect(x,Ax,z) \cap \Vect(x,Ax,z')=\Vect(x,Ax)$.

\vskip 3mm
\noindent \textbf{Case 2.} $N=\lambda\,I_{n-2}$ for some $\lambda \in \K$. \\
\textbf{Subcase 2.1.} $\lambda$ is not an eigenvalue of $K$. \\
Then, $G:=\Ker(A-\lambda I_n)$ has dimension $n-2$.
For $z \in \K^n$, denote by $p_z$ the monic generator of the ideal $\{q \in \K[t] : \; q(A)z=0\}$.
Recall that, given $y$ and $z$ in $\K^n$ for which $p_y$ and $p_z$ are mutually prime,
one has $p_{y+z}=p_y p_z$. In particular, as $p_x$ has degree $2$,
$p_z$ has degree $3$ for every $z \in (\K x\oplus G) \setminus (\K x \cup G)$, that is every
$z$ in $(\K x\oplus G) \setminus (\K x \cup G)$ has order $3$; thus,
$\rk(z,Az,Bz) \leq 2$ for all such $z$. Moreover, it is obvious that $\rk(z,Az,Bz) \leq 2$ for all $z \in G$.

Let us choose a non-zero linear form $\varphi$ on $\K x\oplus G$ such that $\varphi(x)=0$. For every $z \in \K x\oplus G$, set
$$M(z)=\begin{bmatrix}
\varphi(z) & 0 & 0 & 0 \\
[0]_{n \times 1} & z & Az & Bz
\end{bmatrix} \in \Mat_{n+1,4}(\K).$$
Then, with the above results, we know that $\rk M(z) \leq 3$ for all $z \in \K x\oplus G$.
On the other hand, $\max \{\rk \varphi(z) \mid z \in (\K x\oplus G)\}=1$.
Using Lemma \ref{decompositionlemma}, we deduce that
$\rk(z,Az,Bz) \leq 2$ for all $z \in \K x\oplus G$. In particular, $\rk (x,Ax,Bx) \leq 2$.

\vskip 2mm
\noindent \textbf{Subcase 2.2.} $\lambda$ is an eigenvalue of $K$ with multiplicity $1$. \\
Then, there are eigenvectors $y$ and $z$ of $A$, with distinct corresponding eigenvalues, such that $x=y+z$.
Thus, $(y,z)$ may be extended into a basis $\bfB'$ of $\K^n$ such that $P_{\bfB'}^{-1} A P_{\bfB'}$ is upper-triangular.
It follows from point (a) of Lemma \ref{prelimlemma} that $P_{\bfB'}^{-1} B P_{\bfB'}$ is Hessenberg,
and in particular $By \in \Vect(y,z)$. Starting from $(z,y)$ instead of $(y,z)$, one finds $Bz \in \Vect(y,z)$.
Therefore, all the vectors $y+z$, $A(y+z)$ and $B(y+z)$ belong to the $2$-dimensional space $\Vect(y,z)$, which yields
$\rk (x,Ax,Bx) \leq 2$.

\vskip 2mm
\noindent \textbf{Subcase 2.3.} $\lambda$ is an eigenvalue of $K$ with multiplicity $2$ . \\
Then, the characteristic polynomial of $A$ is $(t-\lambda)^n$.
\begin{itemize}
\item Assume that $n \geq 4$. One chooses an eigenvector $y$ of $A$ in $\Vect(x,Ax)$, so that $(y,x)$ is a basis of $\Vect(x,Ax)$.
Then, one chooses an arbitrary non-zero vector $u \in F$, and one extends $(y,x,u)$ into a basis $\bfB'$
of $\K^n$ such that $P_{\bfB'}^{-1} A P_{\bfB'}$ is upper-triangular.
Applying point (a) of Lemma \ref{prelimlemma} once more yields $Bx \in \Vect(y,x,u)=\Vect(x,Ax,u)$.
As $n \geq 4$, we can choose another vector $v \in F \setminus \K u$,
and the above method yields $Bx \in \Vect(x,Ax,v)$, while $x,Ax,u,v$ are linearly independent.
Therefore, $Bx \in \Vect(x,Ax,u) \cap  \Vect(x,Ax,v)=\Vect(x,Ax)$.

\item Finally, assume that $n=3$. As $A$ is not similar to $\lambda I_3+E_{2,3}$,
the only remaining option is that $\rk(A-\lambda I_3)=2$.
Then, we can find a linear form $\varphi$ on $\K^3$ with kernel $\Ker (A-\lambda I_3)^2$.
Every vector $z \in \K^3 \setminus  \Ker (A-\lambda I_3)^2$ has order $3$.
Therefore, for every $z \in \K^3$, either $\varphi(z)=0$ or $\rk (z,Az,Bz)\leq 2$.
With the same line of reasoning as in Subcase 2.1, we obtain $\rk (x,Ax,Bx)\leq 2$.
This completes the proof.
\end{itemize}
\end{proof}

Thus, only two situations are left to consider: the one where $(I_n,A,B)$ is LLD,
and the one where $A$ is similar to $\lambda I_3+E_{2,3}$ for some $\lambda \in \K$.
They are dealt with separately in the next two sections.

\subsection{The case when $(I_n,A,B)$ is locally linearly dependent}

In order to analyze the situation where $(I_n,A,B)$ is LLD, we use the classification of LLD
triples over fields with more than $2$ elements
(this result is found in \cite{dSPLLD2}; prior to that, the result was known for
infinite fields \cite{BresarSemrl} and for fields with more than $4$ elements \cite{ChebotarSemrl}).

\begin{theo}[Classification theorem for LLD triples]\label{LLDclass}
Let $(f,g,h)$ be an LLD triple of linear operators from a vector space $U$ to a vector space $V$, where the underlying field has more than $2$ elements.
Assume that $f,g,h$ are linearly independent and that $\Ker f \cap \Ker g \cap \Ker h=\{0\}$
and $\im f+\im g+\im h=V$.
Then:
\begin{enumerate}[(a)]
\item Either there is a $2$-dimensional subspace $\calP$ of $\Vect(f,g,h)$ and a $1$-dimensional subspace $\calD$
of $V$ such that $\im u \subset \calD$ for all $u \in \calP$;
\item Or $\dim V \leq 2$;
\item Or $\dim U=\dim V=3$ and there are bases of $U$ and $V$ in which the operator space $\Vect(f,g,h)$ is represented by the space
$\Mata_3(\K)$ of all $3 \times 3$ alternating matrices.
\end{enumerate}
\end{theo}

\begin{cor}\label{LLDavecid}
Assume that $\# \K>2$, and let $A$ and $B$ be matrices of $\Mat_n(\K)$, with $n \geq 3$, such that
$(I_n,A,B)$ is LLD. Then, either $I_n,A,B$ are linearly dependent, or
there is a $1$-dimensional subspace $\calD$ of $\K^n$ and scalars $\lambda$ and $\mu$ such that
$\im(A-\lambda I_n)=\calD=\im(B-\mu I_n)$.
\end{cor}

\begin{proof}
Assume that $I_n,A,B$ are linearly independent.
As $\Ker I_n=\{0\}$ and $\im I_n=\K^n$, we are in the position to use Theorem \ref{LLDclass}.
Moreover, $\rk I_n>2$ discards Cases (b) and (c) altogether (as no $3 \times 3$ alternating
matrix is invertible).
Therefore, we have a $2$-dimensional subspace $\calP$ of $\Vect(I_n,A,B)$ and a $1$-dimensional subspace $\calD$ of $\K^n$
such that $\im M \subset \calD$ for all $M \in \calP$. In particular $I_n \not\in \calP$, whence $\Vect(I_n,A,B)=\K I_n \oplus \calP$.
This yields a pair $(\lambda,M_1)\in \K \times \calP$ such that $A=\lambda I_n+M_1$, and hence
$\im(A-\lambda I_n) \subset \calD$. As $A-\lambda I_n \neq 0$ (we have assumed that $I_n,A,B$ are linearly independent),
we deduce that $\im(A-\lambda I_n) =\calD$. Similarly, one finds a scalar $\mu$ such that $\im(B-\mu I_n) =\calD$.
\end{proof}

From there, we can prove the following result as a consequence of Theorem \ref{hypercan}:

\begin{lemma}\label{LLDcase}
Assume that $\# \K >3$ and $n \geq 3$.
Let $(A,B) \in \frak{sl}_n(\K)^2$ be with $B \neq 0$, and set $\calH:=\{B\}^\bot$.
Assume that $(I_n,A,B)$ is LLD and that $A$ is not similar to $\lambda I_3+E_{2,3}$ for some $\lambda \in \K$.
Then, $A \in [\calH,\calH]$.
\end{lemma}

\begin{proof}
We use a \emph{reductio ad absurdum} by assuming that $A \not\in [\calH,\calH]$.
By Corollary \ref{LLDavecid}, we can split the discussion into two main cases.

\noindent \textbf{Case 1.} $I_n,A,B$ are linearly dependent. \\
Assume first that $A \in \K I_n$.
Then, $P^{-1}AP$ is upper-triangular for every $P \in \GL_n(\K)$, and hence Lemma \ref{prelimlemma}
yields that $P^{-1}BP$ is Hessenberg for every such $P$.
In particular, let $x \in \K^n \setminus \{0\}$. For every $y \in \K^n \setminus \K x$,
we can extend $(x,y)$ into a basis $(x,y,y_3,\dots,y_n)$ of $\K^n$, and hence we learn that $Bx \in \Vect(x,y)$.
Using the basis $(x,y_3,y,y_4,\dots,y_n)$, we also find $Bx \in \Vect(x,y_3)$, whence $Bx \in \K x$.
Varying $x$, we deduce that $B \in \K I_n$, whence $\calH=\frak{sl}_n(\K)$. Theorem \ref{hypercan} then
yields $A \in [\calH,\calH]$, contradicting our assumptions.

Assume now that $A \not\in \K I_n$. Then, there are scalars $\lambda$ and $\mu$ such that $B=\lambda A+\mu I_n$.
By Theorem \ref{hypercan}, there are trace zero matrices $M$ and $N$ such that $A=[M,N]$.
Thus $\tr((B-\lambda A)M)=\tr((B-\lambda A)N)=0$. Using principle (2) of Section \ref{proofstrategy}, we
deduce that $(M,N) \in \calH^2$, whence $A \in [\calH,\calH]$.

\vskip 2mm
\noindent \textbf{Case 2.} $I_n,A,B$ are linearly independent. \\
By Corollary \ref{LLDavecid}, there are scalars $\lambda$ and $\mu$ together with a $1$-dimensional subspace $\calD$
of $\K^n$ such that $\im(A-\lambda I_n)=\im(B-\mu I_n)=\calD$.
In particular, $A-\lambda I_n$ has rank $1$, and hence it is diagonalisable or nilpotent.
In any case, $A$ is triangularizable; in the second case, the assumption that $A$ is not similar to $\lambda I_3+E_{2,3}$ leads to $n \geq 4$.

Let $x$ be an eigenvector of $A$. Then, we can extend $x$ into a triple $(x,y,z)$ of linearly independent eigenvectors of $A$
(this uses $n \geq 4$ in the case when $A-\lambda I_n$ is nilpotent).
Then, we further extend this triple into a basis $(x,y,z,y_4,\dots,y_n)$ in which $v \mapsto Av$
is upper-triangular. Point (a) in Lemma \ref{prelimlemma} yields $Bx \in \Vect(x,y)$.
With the same line of reasoning, $Bx \in \Vect(x,z)$, and hence $Bx \in \Vect(x,y)\cap \Vect(x,z)=\K x$.
Thus, we have proved that every eigenvector of $A$ is an eigenvector of $B$. In particular,
$\Ker(A-\lambda I_n)$ is stable under $v \mapsto Bv$, and the resulting endomorphism is a scalar multiple of the identity.
This provides us with some $\alpha \in \K$ such that $(B-\alpha I_n)z=0$ for all $z \in \Ker(A-\lambda I_n)$.
In particular, $\alpha$ is an eigenvalue of $B$ with multiplicity at least $n-1$, and since
$\mu$ shares this property and $n<2(n-1)$, we deduce that $\alpha=\mu$. As $\rk(A-\lambda I_n)=\rk(B-\mu I_n)=1$, we
deduce that $\Ker(A-\lambda I_n)=\Ker(B-\mu I_n)$. Thus, $A-\lambda I_n$ and $B-\mu I_n$ are two rank $1$ matrices
with the same kernel and the same range, and hence they are linearly dependent. This contradicts the assumption that $I_n,A,B$
be linearly independent, thereby completing the proof.
\end{proof}

\subsection{The case when $A=\lambda I_3+E_{2,3}$}\label{specialcasesection}

\begin{lemma}\label{finalcase}
Assume that $\# \K>2$. Let $\lambda \in \K$. Assume that
$A:=\lambda I_3+E_{2,3}$
has trace zero. Let $B \in \frak{sl}_3(\K) \setminus \{0\}$, and set
$\calH:=\{B\}^\bot$. Then, $A \in [\calH,\calH]$.
\end{lemma}

\begin{proof}
We assume that $A \not\in [\calH,\calH]$ and search for a contradiction.
By point (a) in Lemma \ref{prelimlemma}, for every basis $\bfB=(x,y,z)$ of $\K^3$
for which $P_\bfB^{-1}\,A\,P_\bfB$ is upper-triangular, we find $Bx \in \Vect(x,y)$.
In particular, for every basis $(x,y)$ of $\Vect(e_1,e_2)$, the triple $(x,y,e_3)$ qualifies,
whence $Bx \in \Vect(x,y)=\Vect(e_1,e_2)$. It follows that $\Vect(e_1,e_2)$ is stable under $B$.
As $z \mapsto Az$ is also represented by an upper-triangular matrix in the basis $(e_2,e_3,e_1)$, one finds $Be_2 \in \Vect(e_2,e_3)$, whence
$Be_2 \in \K e_2$.
Thus, $B$ has the following shape:
$$B=\begin{bmatrix}
a & 0 & d \\
b & c & e \\
0 & 0 & f
\end{bmatrix}.$$
From there, we split the discussion into two main cases.

\noindent \textbf{Case 1.} $\lambda = 0$. \\
Using $(e_2,e_1,e_3)$ as our new basis, we are reduced to the case when
$$A=\begin{bmatrix}
0 & 0 & 1 \\
0 & 0 & 0 \\
0 & 0 & 0
\end{bmatrix} \quad \text{and} \quad
B=\begin{bmatrix}
? & ? & ? \\
0 & ? & ? \\
0 & 0 & ?
\end{bmatrix}.$$
Then, one checks that $[J_2,E_{2,3}]=A$, and $\tr(J_2B)=0=\tr(E_{2,3}B)$.
This yields $A \in [\calH,\calH]$, contradicting our assumptions.

\vskip 2mm
\noindent \textbf{Case 2.} $\lambda \neq 0$. \\
As we can replace $A$ with $\lambda^{-1}A$, which is similar to $I_3+E_{2,3}$, no generality is lost in assuming that $\lambda=1$.
According to principle (2) of Section \ref{proofstrategy}, no further generality is lost
in subtracting a scalar multiple of $A$ from $B$, to the effect that we may assume that $f=0$
and $B \neq 0$ (if $B$ is a scalar multiple of $A$, then the same principle combined with the Albert-Muckenhoupt theorem
shows that $A \in [\calH,\calH]$). As $\tr B=0$, we find that
$$B=\begin{bmatrix}
a & 0 & d \\
b & -a & e \\
0 & 0 & 0
\end{bmatrix}.$$
Note finally that $\K$ must have characteristic $3$ since $\tr A=0$.

\noindent \textbf{Subcase 2.1.} $b \neq 0$. \\
As the problem is unchanged in multiplying $B$ with a non-zero scalar, we can assume that $b=1$.
Assume furthermore that $d \neq 0$.
Let $(\alpha,\beta)\in \K^2$, and set
$$C:=\begin{bmatrix}
0 & 1 & 0 \\
\alpha & 0 & 1 \\
\beta & 0 & 0
\end{bmatrix}.$$
Note that $C$ is a cyclic matrix and
$$C^2=\begin{bmatrix}
\alpha & 0 & 1 \\
\beta & \alpha & 0 \\
0 & \beta & 0
\end{bmatrix}.$$
Thus, $\tr(AC)=0$, $\tr(BC)=\beta d+1$, $\tr(AC^2)=2\alpha+\beta=\beta-\alpha$
and $\tr(BC^2)=e \beta$. As $d \neq 0$, we can set
$\beta:=-d^{-1}$ and $\alpha:=\beta$, so that $\beta\neq 0$ and
 $\tr(A)=\tr(AC)=\tr(AC^2)=0$. Thus, $A \in \im (\ad_C)$ by Lemma \ref{imadcyclic}, and on the other hand $C \in \calH$. As $A \not\in [\calH,\calH]$, it follows that $\calC(C) \subset \calH$, and hence $\tr(BC^2)=0$. As $\beta \neq 0$, this yields $e=0$.

From there, we can find a non-zero scalar $t$ such that $d+t\,a \neq 0$ (because $\# \K>2$).
In the basis $(e_1,e_2,e_3+t\,e_1)$, the respective matrices of $z \mapsto Az$ and $z \mapsto Bz$ are
$I_3+E_{2,3}$ and
$$\begin{bmatrix}
a & 0 & d+t\,a \\
1 & -a & t \\
0 & 0 & 0
\end{bmatrix}.$$
As $d+t\,a \neq 0$ and $t \neq 0$, we find a contradiction with the above line of reasoning.

Therefore, $d=0$.
Then, the matrices of $z \mapsto Az$ and $z \mapsto Bz$ in the basis $(e_1,e_2,e_3+e_1)$
are, respectively, $I_3+E_{2,3}$ and $\begin{bmatrix}
a & 0 & a \\
1 & -a & e+1 \\
0 & 0 & 0
\end{bmatrix}$. Applying the above proof in that new situation yields $a=0$.
Therefore,
$$B=\begin{bmatrix}
0 & 0 & 0 \\
1 & 0 & e \\
0 & 0 & 0
\end{bmatrix}$$
With $(e_3-e\,e_1,e_1,e_2)$ as our new basis, we are finally left with the case when
$$A=\begin{bmatrix}
1 & 0 & 0 \\
0 & 1 & 0 \\
1 & 0 & 1
\end{bmatrix} \quad \text{and} \quad B=\begin{bmatrix}
0 & 0 & 0 \\
0 & 0 & 0 \\
0 & 1 & 0
\end{bmatrix}.$$
Set
$$C:=\begin{bmatrix}
1 & 0 & 1 \\
1 & 1 & 0 \\
0 & 1 & 0
\end{bmatrix}$$
and note that $C$ is cyclic and
$$C^2=\begin{bmatrix}
1 & 1 & 1 \\
-1 & 1 & 1 \\
1 & 1 & 0
\end{bmatrix}.$$
One sees that $\tr(A)=\tr(AC)=\tr(AC^2)=0$, and hence $A \in \im (\ad_C)$ by Lemma \ref{imadcyclic}.
On the other hand, $\tr(BC)=0$. As $A \not\in [\calH,\calH]$, one should find $\tr(BC^2)=0$,
which is obviously false. Thus, we have a final contradiction in that case.

\noindent \textbf{Subcase 2.2.} $b=0$. \\
Assume furthermore that $a \neq 0$. Then, in the basis $(e_1+e_2,e_2,e_3)$, the respective matrices of $z \mapsto Az$ and $z \mapsto Bz$
are $I_3+E_{2,3}$ and $\begin{bmatrix}
a & 0 & d \\
-2a & -a & e-d \\
0 & 0 & 0
\end{bmatrix}$. This sends us back to Subcase 2.1, which leads to another contradiction.
Therefore, $a=0$.

If $d=0$, then we see that $B \in \Vect(I_n,A)$, and hence principle (2) from Section \ref{proofstrategy} combined
with Theorem \ref{hypercan} shows that $A \in [\calH,\calH]$, contradicting our assumptions.
Thus, $d \neq 0$. Replacing the basis $(e_1,e_2,e_3)$ with $(d\,e_1+e\,e_2,e_2,e_3)$, we
are reduced to the case when
$$A=\begin{bmatrix}
1 & 0 & 0 \\
0 & 1 & 1 \\
0 & 0 & 1
\end{bmatrix} \quad \text{and} \quad
B=\begin{bmatrix}
0 & 0 & 1 \\
0 & 0 & 0 \\
0 & 0 & 0
\end{bmatrix}.$$
In that case, we set
$$C:=\begin{bmatrix}
0 & 0 & 0 \\
1 & 0 & 0 \\
0 & 1 & -1
\end{bmatrix}$$
which is a cyclic matrix with
$$C^2=\begin{bmatrix}
0 & 0 & 0 \\
0 & 0 & 0 \\
1 & -1 & 1
\end{bmatrix},$$
so that $\tr(A)=\tr(AC)=\tr(AC^2)=0$ and $\tr(BC)=0$. As $\tr(BC^2) \neq 0$,
this contradicts again the assumption that $A \not\in [\calH,\calH]$.
This final contradiction shows that the initial assumption $A \not\in [\calH,\calH]$ was wrong.
\end{proof}

\subsection{Conclusion}

Let $A \in \Mat_n(\K)$ and $B \in \Mat_n(\K)\setminus \{0\}$, where $n \geq 3$ and $\# \K \geq 4$.
Set $\calH:=\{B\}^\bot$ and assume that $\tr(A)=0$ and $\tr(B)=0$.
If $A$ is similar to $\lambda I_3+E_{2,3}$, then we know from Lemma \ref{finalcase}
and principle (3) of Section \ref{proofstrategy} that $A \in [\calH,\calH]$.
Otherwise, if $(I_n,A,B)$ is LLD then we know from Lemma \ref{LLDcase} that $A \in [\calH,\calH]$.
Using Lemma \ref{reductionlemma}, we conclude that $A \in [\calH,\calH]$ in every possible situation. This completes
the proof of Theorem \ref{dSPcrochet}.

\end{document}